\documentclass{amsart}
\pdfoutput=1

\usepackage{amssymb}
\usepackage{microtype}
\usepackage{tikz}
\usetikzlibrary{arrows}
\usepackage{booktabs}
\usepackage[pdftex,colorlinks,citecolor=black,linkcolor=black,urlcolor=black,bookmarks=false]{hyperref}
\def\MR#1{\href{http://www.ams.org/mathscinet-getitem?mr=#1}{MR#1}}

\usepackage{doi}
\usepackage{subfigure}
\usepackage{bbm} 

\newtheorem{theorem}{Theorem}[section]
\newtheorem{proposition}[theorem]{Proposition}
\newtheorem{lemma}[theorem]{Lemma}
\newtheorem{corollary}[theorem]{Corollary}

\theoremstyle{definition}

\numberwithin{figure}{section}
\numberwithin{equation}{section}
\numberwithin{table}{section}

\newcommand{\N}{\mathbb{N}}
\newcommand{\Z}{\mathbb{Z}}
\newcommand{\R}{\mathbb{R}}
\newcommand{\C}{\mathbb{C}}
\newcommand{\F}{\mathbb{F}}

\newcommand{\A}{\mathbb{A}}

\newcommand{\supp}{\operatorname{supp}}
\newcommand{\GL}{{\rm GL}}
\newcommand{\PGL}{{\rm PGL}}
\newcommand{\SL}{{\rm SL}}

\newcommand{\Fix}{{\rm Fix}}

\newcommand{\Hom}{{\rm Hom}}
\newcommand{\Aut}{{\rm Aut}}
\newcommand{\Out}{{\rm Out}}

\newcommand{\rot}{{\rm rot}}

\DeclareMathOperator{\tr}{tr}
\newcommand{\legendre}[2]{\genfrac{(}{)}{}{}{#1}{#2}}

\title{Kesten-McKay law for the Markoff surface mod $p$}

\author{Matthew de Courcy-Ireland}
\address{Department of Mathematics\\
Princeton University\\
Princeton, NJ, 08544, USA} \email{mdc4@math.princeton.edu}

\author{Michael Magee}
\address{Department of Mathematical Sciences \\
Durham University \\
Durham, DH1 3LE, UK} \email{michael.r.magee@durham.ac.uk}

\date{October 31, 2018}

\begin{document}

\begin{abstract}
For each prime $p$, we study the eigenvalues of a 3-regular graph on roughly $p^2$ vertices constructed from the Markoff surface. 
We show they asymptotically follow the Kesten-McKay law, which also describes the eigenvalues of a random regular graph. The proof is based on the method of moments and takes advantage of a natural group action on the Markoff surface.
\end{abstract}

\maketitle

\section{Introduction} \label{sec:introduction}

The Kesten-McKay Law governs the eigenvalue distribution of a random $d$-regular graph in the limit of a growing number of vertices \cite{M}, \cite{K}. 
The limiting probability density function is
\begin{equation}
\rho_d(\lambda) = \frac{d}{2\pi} \frac{ \sqrt{4(d-1) - \lambda^2} }{d^2 - \lambda^2} \mathbbm{1}_{[-2\sqrt{d-1},2\sqrt{d-1}]}(\lambda)
\end{equation}
This spectral density comes from the Plancherel measure on the infinite $d$-regular tree, and one might expect a similar eigenvalue distribution for non-random $d$-regular graphs provided they resemble their universal cover closely enough in the sense of having few short cycles.
The purpose of this article is to establish such a result for a family of 3-regular graphs constructed from the \emph{Markoff equation}
\begin{equation} \label{eqn:markoff}
 x^2+y^2+z^2=xyz
\end{equation}
modulo large prime numbers $p \rightarrow \infty$. The vertices, roughly $p^2$ in number, are simply the solutions $(x,y,z)$ in $\F_p^3$ excluding $(0,0,0)$. The edges connect $(x,y,z)$ to $(x,y,xy-z)$, $(x, xz-y, z)$, and $(yz-x,y,z)$, the Markoff equation being preserved by these operations.
We will write $M(\F_p)$ for the vertex set and $\mathfrak{M}(\F_p)$ for the graph.
The eigenvalues $\{ \lambda_j \}$ of the resulting graph can naturally be thought of as a measure on $[-3,3]$, namely
\begin{equation} \label{eqn:empirical}
\mu_p = \frac{1}{|M(\F_p)|} \sum_j \delta_{\lambda_j}
\end{equation}
and our main result is that the moments of this measure converge as $p \rightarrow \infty$ to those of the Kesten-McKay measure.
\begin{theorem}
There is an absolute $C > 1$ such that, with an implicit constant independent of both $p$ and $L$,
\[
\int x^L d\mu_p = \int x^L \rho_3(x) dx + O\left( \frac{ C^L }{p} \right).  
\]
\label{thm:logp}
\end{theorem}
Thus one can take $L$ to be a small multiple of $\log{p}$ and the error term $C^L/p$ will remain negligible.
Our proof of Theorem~\ref{thm:logp} permits $C = 2^{17} = 131072$, which we have not optimized, but an exponential dependence on $L$ is inevitable. As we will explain at the end of the paper, the moments no longer agree if $L/\log{p}$ is too large.

Taking linear combinations and applying Theorem~\ref{thm:logp} for $L$ fixed as $p \rightarrow \infty$, we obtain
\begin{theorem}
For any fixed polynomial $f$, the eigenvalues $\lambda_j$ of the Markoff graph mod $p$ satisfy
\[
\frac{1}{p^2\pm 3p} \sum_{j} f(\lambda_j) = \int_{-2\sqrt{2}}^{2\sqrt{2}} f(\lambda) \rho_3(\lambda) d\lambda + O\left( \frac{1}{p} \right)
\]
as $p \rightarrow \infty$.
\label{thm:boundedL}
\end{theorem}

\begin{figure}[t] 
\centering
\hfill
\subfigure[$p=83$]{\includegraphics[width=0.48\textwidth]{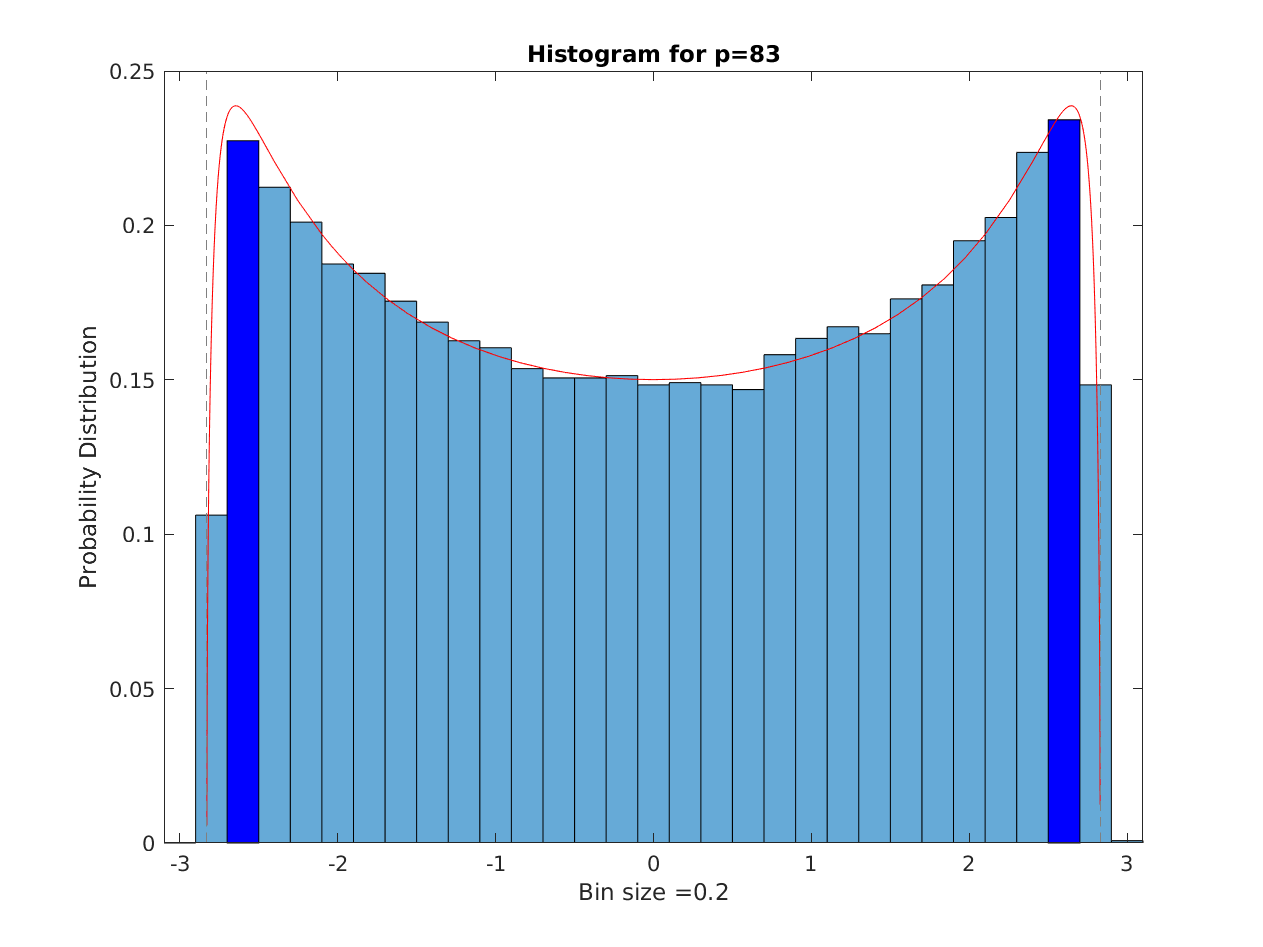}}
\hfill
\subfigure[$p=89$]{\includegraphics[width=0.48\textwidth]{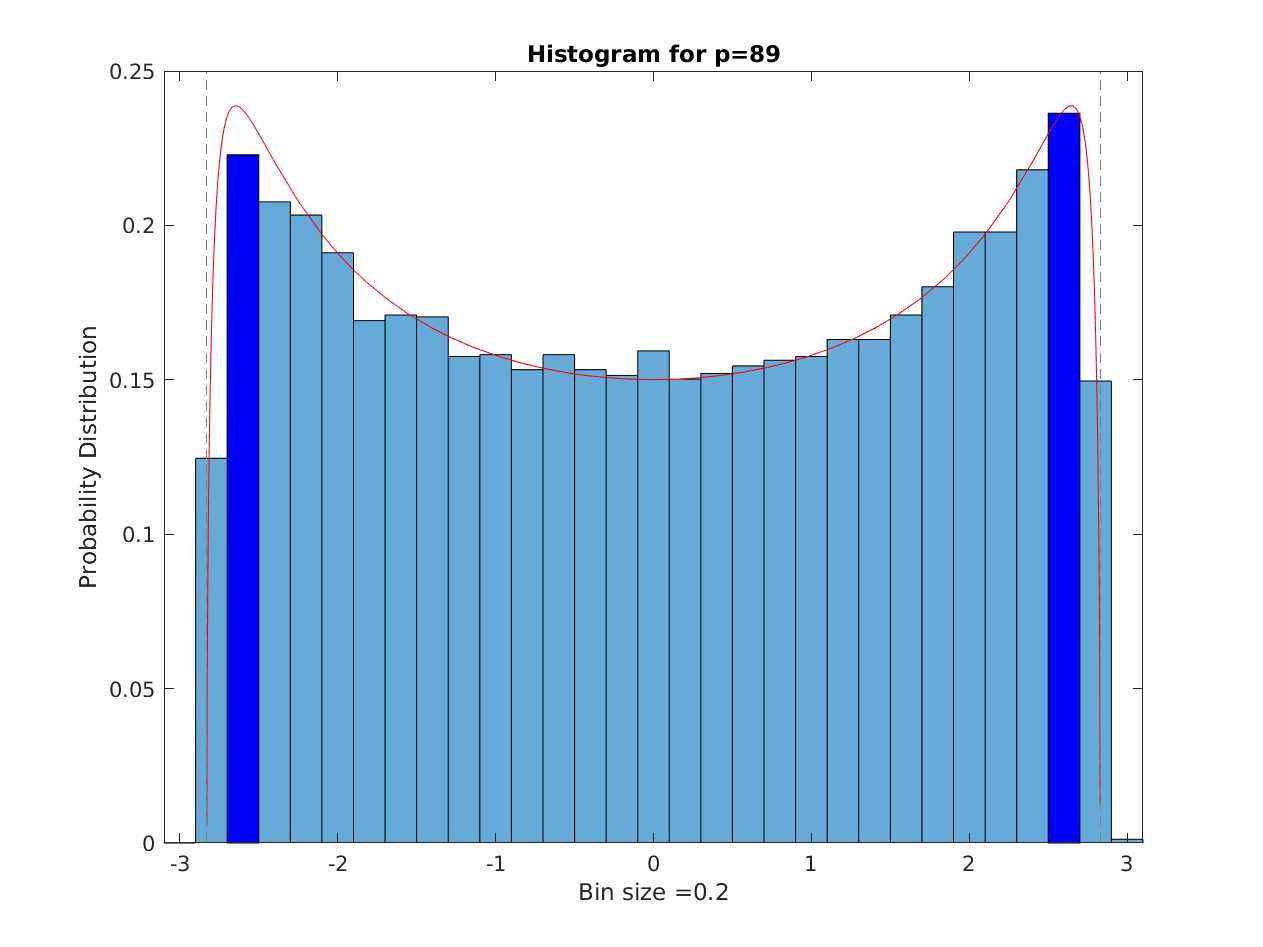}}
\hfill
\caption{Histogram of eigenvalues for $p=83$ and $89$ with the density $\rho_3(x)$ shown in red.}
\label{fig:histogram}
\end{figure}

Taking $L$ growing simultaneously with $p$ gives much more information than one could achieve from any fixed $L$. In particular, we deduce the following bound for the discrepancy between $\mu_p$ and $\rho_3$.
\begin{corollary}
For any interval $I \subseteq [-3,3]$,
\[
| \mu_p(I) - \rho_3(I) | \lesssim \frac{1}{\log{p}}.
\]
\label{cor:discrepancy}
\end{corollary}
Note that the $d$-regular Kesten-McKay measure is supported on the interval $[-2\sqrt{d-1}, 2\sqrt{d-1}]$. A $d$-regular graph could in principle have eigenvalues throughout the interval $[-d,d]$, but when the Kesten-McKay law is valid, it implies that most of the eigenvalues lie in a much smaller interval. As an application of Corollary~\ref{cor:discrepancy}, we have
\begin{corollary}
For the Markoff graph mod $p$, the number of eigenvalues greater than $2\sqrt{2}$ is only $O(p^2/\log{p})$ out of a total of $p^2 \pm 3p$.
\label{cor:exceptional-eigs}
\end{corollary}
It is plausible that one could replace $1/\log{p}$ by $1/p$ in both of these corollaries.
To use our estimates for moments, we approximate the discontinuous indicator function by polynomials, and this entails some loss.

Figure \ref{fig:histogram} shows the histogram of eigenvalues for the Markoff graphs constructed from $p=83$ and $p=89$, illustrating the fit to the Kesten-McKay law. For $3$-regular graphs, the support is $[-2\sqrt{2}, 2\sqrt{2}]$ and the distribution is bimodal, with maxima at $\pm \sqrt{7}$. 

We begin in Section~\ref{sec:graph} with the overall strategy of comparing the Kesten-McKay moments with those of the graphs we construct. This reduces the problem to counting the fixed points of a natural group action. 
In Section~\ref{sec:examples}, we compute the fixed points in several examples and outline a heuristic that would give a better dependence on $L$ in Theorem~\ref{thm:logp}.
In Section~\ref{sec:fricke}, we review the connection between the Markoff surface and $\GL_2(\Z)$, which is the basis for the actual proof. 
In Section~\ref{sec:torsion}, we review the structure of the group generated by the Markoff moves and in particular identify the torsion elements.
In Theorem~\ref{thm:fixed-points}, we prove that an element has $O(p)$ fixed points with an implicit constant depending on its entries as a matrix in $\GL_2(\Z)$.
In Section~\ref{sec:wrap-up}, we complete the proof of Theorem~\ref{thm:logp} by noting that the matrix entries are exponential in the length $L$ of the word.
In Section~\ref{sec:discrepancy}, we turn to the proof of Corollaries \ref{cor:discrepancy} and \ref{cor:exceptional-eigs}.
In Section~\ref{sec:katz?}, we study a related question about the Markoff group in characteristic $p$, showing that the reduction map from characteristic 0 is injective. We prove
\begin{proposition} \label{prop:katz}
Over any algebraically closed field $F$, let $G$ be the group of polynomial automorphisms $\A^3 \mapsto \A^3$ generated by the Markoff moves. Then $G$ is a free product
\[
G \cong \Z/2 * \Z/2 * \Z/2.
\]
\end{proposition}
This is well known when $F = \C$. For $F$ an algebraic closure of $\F_p$, the question is suggested by our estimates for fixed points: A nontrivial relation between the Markoff moves in characteristic $p$ would lead to a word fixing the entire surface. 
This issue does not arise when $L$ is small compared to $p$, and Proposition~\ref{prop:katz} is not strictly necessary for the proof of Theorem~\ref{thm:logp}.
Nevertheless, we give a proof, which makes use of other level sets $x^2+y^2+z^2 = xyz + k$ besides $k=0$.
In Section~\ref{sec:conc}, we conclude by comparing the Kesten-McKay law to other (more difficult) questions about the graphs $\mathfrak{M}_p$, in particular their connectedness and spectral gap.
We use the rest of this Introduction to summarize some of the recent interest in the Markoff equation and its solutions modulo $p$.

The original Markoff surface is defined by the cubic equation
\begin{equation} \label{eqn:markoff}
x^2+y^2+z^2=3xyz
\end{equation}
and its solutions in nonnegative integers $(x,y,z)$ are called Markoff triples.
It differs from our normalization $x^2+y^2+z^2 = xyz$ by a scaling $(x,y,z) \mapsto 3(x,y,z)$, which is invertible over $\F_p$ for $p \geq 5$.
The Markoff equation is a very special case which offers a great simplification compared to other cubic surfaces. The only cubic term in equation (\ref{eqn:markoff}) is $3xyz$, so upon fixing two variables, it is only a quadratic equation for the third. Exchanging the two roots of this quadratic allows us to move from one triple to another. By Vieta's Rule, the two solutions of a quadratic must add up to its middle coefficient, so one such move sends $(x,y,z)$ to another Markoff triple $(x,y,3xy -z)$. There is another move for each of the variables. 
Markoff proved in 1880 \cite{M} that any Markoff triple except $(0,0,0)$ can be reached starting from the solution $(1,1,1)$ by a sequence of Vieta operations and transpositions. 
In contrast, for a general cubic surface, there is no known method for deciding whether there are integer solutions, let alone finding all of them. For instance, it remains out of reach to determine whether a given number is a sum of three (possibly negative) cubes.

The Markoff triples can be displayed as a $3$-regular tree, with $(1,1,1)$ as the root and edges giving the action of the Vieta moves.
Reducing this Markoff tree modulo a prime $p$ yields the finite graph with cycles that we investigate below. In principle, there may be additional solutions over $\F_p$ that do not come from reducing integer solutions mod $p$.  Hence it is no longer guaranteed that all solutions can be found by the Vieta moves, although in practice it seems that they can. If every solution mod $p$ lifts to a solution over the integers, then the same sequence of Vieta moves used to reach the lift will reach its image mod $p$ because the moves are polynomial operations in $(x,y,z)$. Thus the graph of solutions over $\F_p$ will be connected. The connectedness of these graphs for all $p$ is the question of whether \emph{strong approximation} holds for equation~(\ref{eqn:markoff}), that is, whether solutions mod $p$ can always be lifted to integer solutions.
Baragar was the first to conjecture that this connectedness does hold for all $p$ and he verified it for $p \leq 179$ (see p. 124 of \cite{Bar}).

Bourgain-Gamburd-Sarnak \cite{BGS} proved that, for most primes $p$, there is only a single component of nonzero solutions $(x,y,z) \neq (0,0,0)$. Their method fails in case $p^2-1$ has many prime factors, which happens only for rare values of $p$. Even for these exceptional primes, the Bourgain-Gamburd-Sarnak argument shows that there is a giant component containing, for any given $\varepsilon > 0$, all but $p^{\varepsilon}$ of the vertices, while any putative extra components would have size at least a power of $\log(p)$.
On the quantitative level, some improvements have been made by Konyagin-Makarychev-Shparlinski-Vyugin (\cite{KMSV}, Theorems 1.3 and 1.4).
Meiri-Puder \cite{MPC} prove that the Markoff action on the largest component is highly transitive: up to grouping solutions by sign changes as in (\ref{eqn:signs}) below, it is either the full symmetric group or its alternating subgroup. Cerbu-Gunther-Magee-Peilen \cite{CGMP} had proposed earlier that the alternating group arises when $p \equiv 3 \bmod 16$, and the full symmetric group otherwise.

\section{Method of moments} \label{sec:graph}

Let us define the Markoff graph over $\F_p$ more precisely. The vertices are the triples $(x,y,z)$ solving $x^2 + y^2 + z^2 = xyz$, except $(0,0,0)$. The most natural graph for our purposes is defined by taking an edge between $(x,y,z)$ and each of its images $(x,y,xy-z)$, $(x,xz-y,z)$, and $(yz-x,y,z)$. 
We denote the graph by $\mathfrak{M}_p$ and its vertex set by $M(\F_p)$.
It has $p^2 \pm 3p$ vertices depending on whether $p$ is congruent to 1 or to 3 modulo 4. The total number of solutions to $x^2+y^2+z^2=xyz$ mod $p$ is $p^2 + 3\legendre{-1}{p}p + 1$, but we consider $(0,0,0)$ separately from the other solutions because it is in an orbit of its own under the Markoff moves. See equation (2) of Carlitz's note \cite{C} for this count.

At present, we have no guarantee that this graph is connected. Baragar \cite{Bar} conjectured that $\mathfrak{M}_p$ is connected for any prime $p$, and Bourgain-Gamburd-Sarnak proved connectedness unless $p^2-1$ has many small factors in a quantified way \cite{BGS}. They also prove that, even in a possibly disconnected case, there is a giant component containing at least $p^2 \pm 3p - O(p^{\varepsilon})$ vertices for any $\varepsilon > 0$. Our Theorem~\ref{thm:logp} applies both to the whole graph, possibly disconnected, and also to its giant component.

The graphs we study are not simple: Although $M(\F_p)$ does not contain multiple edges, there are loops at a small fraction of the vertices. On the order of $p$ vertices out of $p^2$ have loops. We discuss this further in Proposition~\ref{prop:examples}, and the presence of loops appears again in Lemma~\ref{lem:torsion-fix}. It has some importance for our main proofs.

The graph $M(\F_p)$ is obtained directly from the underlying symmetry of the equation $x^2+y^2 +z^2 =xyz$ under the Markoff moves $m_1$, $m_2$, $m_3$. 
Sometimes it may be preferable to take other edges reflecting further symmetries of the Markoff surface. 
The Markoff equation is preserved by all permutations of $(x,y,z)$ as well as the four \emph{double sign changes} leaving $xyz$ invariant, namely
\begin{equation}
(x,y,z) \mapsto (\sigma_1 x, \sigma_2 y, \sigma_3 z)
\label{eqn:signs}
\end{equation}
where the signs obey $\sigma_1 \sigma_2 \sigma_3 = 1$.
One could add edges corresponding to any of these. Or one could streamline the graph by first taking the quotient by sign changes, or using alternative generators that combine the Markoff moves with permutations. In this way, one could obtain graphs with fewer loops and a closer fit to the Kesten-McKay law. Nevertheless, the Markoff moves themselves seemed the most natural choice to us.

Let $A$ be the adjacency matrix for the Markoff graph mod $p$, that is, the matrix indexed by vertices with $A_{ij} = 1$ when there is an edge between $i$ and $j$ and $A_{ij}=0$ otherwise. 
Note that the diagonal entries $A_{jj}$ are typically 0, but may be 1 when there is a loop connecting $j$ to $j$.
Permuting the vertices changes the adjacency matrix to $\sigma A \sigma^{-1}$, where $\sigma$ is the corresponding permutation matrix. Thus the eigenvalues of $A$ do not depend on any choice of ordering.
The connectedness of a graph is closely related to its eigenvalues. Indeed, for a $d$-regular graph, the number of connected components is the multiplicity of $d$ as an eigenvalue.
The Kesten-McKay law is a general theorem about the distribution of eigenvalues for graphs with few short cycles, either random or deterministic.
We quote the following theorem of McKay (\cite{M}, Theorem 1.1) to emphasize the generality of the Kesten-Mckay law, although we will not be able to use this version to deduce the rate of convergence in Theorem~\ref{thm:logp}.

\begin{theorem} (McKay)
If $G_i$ is a sequence of $d$-regular graphs with $n_i$ vertices such that for each fixed $k$, the number of $k$-cycles in $G_i$ is $o(n_i)$ as $n_i \rightarrow \infty$, then the eigenvalue counting function
\[
\frac{ \# \{ j ; \ \lambda_j(G_i) \leq \lambda \}}{n_i}
\]
converges to
\begin{equation*}
\int_{-\infty}^{\lambda} \frac{d}{2\pi} \frac{ \sqrt{4(d-1) - t^2} }{d^2 - t^2} \mathbbm{1}_{[-2\sqrt{d-1},2\sqrt{d-1}]}(\lambda) dt
\end{equation*}
as $n_i \rightarrow \infty$.
\label{thm:mckay}
\end{theorem}

The combinatorial significance of the Kesten-McKay measure is that its moments count walks in a $d$-regular tree
\[
\int_{-2\sqrt{d-1}}^{2\sqrt{d-1}} x^L \rho_d(x) dx = \#\big( \text{closed walks of length} \ L \big)
\]
where the walks must start and return at a designated root of the tree.

For the case of the Markoff graph mod $p$, the number of vertices is $p^2 \pm 3p$. Thus all we have to show is that the number of $k$-cycles is $o(p^2)$ for each fixed $k$. 
For intuition, imagine proving McKay's theorem by the method of moments. 
The moments are given by
\[
\tr(A^L) = \sum \lambda_j^L
\]
up to normalization by $p^2 \pm 3p$.
On the other hand, there is a combinatorial interpretation.
For $L \geq 1$, the trace $\tr(A^L)$ counts paths of length $L$ in the graph:
\begin{align*}
\tr(A^L) &= \sum_{j} \sum_{k_1} \cdots \sum_{k_l} a_{jk_1}a_{k_1k_2} \ldots a_{k_L j} \\ 
&= \sum_{x \in M(\F_p)} \sum_{x \overset{L}{\rightarrow} x } 1
\end{align*}
where the inner sum runs over paths of length $L$ from $x$ to $x$, and the outer sum runs over all vertices $x$. Changing the order of summation, we can rewrite this as
\[
\tr(A^L) = \sum_w \sum_{\gamma: I \rightarrow w} \#\{ {\rm fixed \ points \ of \ } w \}
\]
In the outer sum, $w$ is a reduced word of length $L$ in the free product $\Z/2 * \Z/2 * \Z/2$ with generators $m_1, m_2, m_3$.
In the inner sum, $\gamma$ is a path from the identity to $w$. 
Note that if $w = I$ is the identity, then all of the $p^2 \pm 3p$ vertices are fixed points, making a contribution of 
\[
(\#\text{paths to and from } I )(p^2 + O(p) ).
\]
We divide by $p^2$ for normalization, and the remaining path-count is exactly the corresponding Kesten-McKay moment.

\section{Some examples and heuristics} \label{sec:examples}

To argue that the trivial word contributes the main term, we must study the fixed points of other words in the Markoff moves. Let $w = g_1 \cdots g_L$ be a reduced word of length $L$ where each $g_i$ is one of the Markoff moves $m_1, m_2, m_3$. Write the fixed point equation as
\begin{equation}
\begin{bmatrix}
f(x,y,z) \\
g(x,y,z) \\
h(x,y,z) \\
x^2+y^2+z^2-xyz
\end{bmatrix}
=
\begin{bmatrix}
x \\
y \\
z \\
0
\end{bmatrix}
\label{eqn:fixed}
\end{equation}
where $f$, $g$, and $h$ are polynomials that can be computed by successively applying the moves that make up the word $w$. One might expect this system of four equations in only three unknowns to have no solutions, but there may be redundancy. Indeed, the system always has $(0,0,0)$ as a trivial solution. The extreme case is $w=1$, for which the first three equations amount to $(x,y,z)=(x,y,z)$ and every point on the Markoff surface is fixed. For nontrivial words, we will use the special structure of the Markoff surface to show that there is at least one nontrivial constraint in addition to the equation $x^2+y^2+z^2=xyz$. First, we consider a few examples of short words.

\begin{proposition} \label{prop:examples}
(Fixed points of short words)
\begin{enumerate}
\item
The number of fixed points of a single Markoff move $m_i$ is $p - 4 - \legendre{-1}{p}$, and in particular is at most $p$.
\item
A reduced word of length 2 has no fixed points.
\item
A reduced word of length 3 either has no fixed points or else is conjugate to a single Markoff move.
\end{enumerate}
\end{proposition}

Part (1) shows that the graph $\mathfrak{M}_p$ contains loops, but only at a small fraction of the vertices. This example also shows that it is possible for (\ref{eqn:fixed}) to reduce to just one nontrivial constraint in addition to the Markoff equation. The fact that the words of length 1 together have only on the order of $p$ fixed points has some importance for our main proofs and we will revisit it in Lemma~\ref{lem:torsion-fix}.
Part (2) shows that there are never multiple edges joining the same pair of vertices.
Part (3) shows that the graph contains triangles.

\emph{Proof of (1)}
This count is given in Lemma 2.3 of \cite{CGMP}, noting that $p-4 - \legendre{-1}{p}$ is $p-5$ when $p\equiv1\bmod4$ and $p-3$ when $p\equiv3\bmod4$. For the reader's convenience, we sketch a similar argument here.
If $(x,y,z)  = (x,y,xy-z)$, then the Markoff move $m_3$ connects the vertex $(x,y,z)$ to itself. Substituting $z = xy-z$ into the Markoff equation gives
\[
x^2 + y^2 + \left( \frac{xy}{2} \right)^2 = \frac{x^2y^2}{2}.
\]
For each $y \in \F_p$, this is a quadratic equation for $x$, namely
\[
(y^2 - 4)x^2 = (2y)^2
\]
which has no solutions if $y^2 = 4$, a unique soluton $x=0$ in case $y=0$, and otherwise has $1 + \legendre{y^2-4}{p}$ solutions. 
If $y = 0$, then the fixed point must be $(0,0,0)$, which is not part of our graph.
Thus we remove it from the count and find that the number of solutions is
\[
\sum_{y \neq \pm 2} \left(1 + \legendre{y^2-4}{p} \right) - 1 - \legendre{-1}{p} = p-3 - \legendre{-1}{p} + \sum_{y \neq \pm 2} \legendre{y^2-4}{p}.
\]
The character sum can be evaluated by factoring $y^2-4$ as $(y-2)(y+2)$, changing variables to $u = y-2$, and using $\legendre{u^{-1}}{p} = \legendre{u}{p}$. Note that $v = 1 + 4/u$ assumes all values except $1$ and $0$ when $u$ is restricted to $u \neq -4,0$, so that
\[
\sum_{y \neq \pm 2 } \legendre{y^2-4}{p} = \sum_{u \neq -4,0} \legendre{u}{p} \legendre{u+4}{p} = \sum_{v \neq 1,0} \legendre{v}{p} = -1.
\]
Our count becomes $p-4 - \legendre{-1}{p}$ and the result follows. \qed

For any $(x,y)$ solving $y^2 = x^2(y^2/4-1)$, taking $z = xy/2$ gives a point $(x,y,z)$ connected to itself by $m_3$. In the same way, taking $x = yz/2$ or $y=xz/2$ gives points fixed by $m_1$ or $m_2$. All told, there are $3(p-4-\legendre{-1}{p})$ vertices fixed by one of the generators. At each such vertex, there is a single loop. 
Note, as a special case of part (2), that only $(0,0,0)$ is fixed by multiple generators at once.

\emph{Proof of (2)}
A word of length 2 has no fixed points.
We stated before that the Markoff graph does not contain bigons -- that is, multiple edges between the same pair of vertices -- and a fixed point $x$ of $m_i m_j$ is equivalent to a bigon between $x$ and $m_j x$. It is easy to see why this does not occur. For example, if the moves $m_3$ and $m_2$ define the same edge starting from $(x,y,z)$, then
\[
(x,y,xy-z) = (x, xz-y, z).
\]
Equivalently, $2y = xz$ and $2z = xy$. Thus $2y = x^2y/2$, which implies that either $y=0$ or $x = \pm 2$.
If $y=0$, then $2z = xy = 0$ forces $z=0$, and then the Markoff equation implies that $x$ is also 0. Thus this case arises only for $(x,y,z) = (0,0,0)$, which is not part of our graph.
On the other hand, the cases $x = \pm 2$ do not arise at all. Indeed, $2z = xy = \pm 2y$ implies $z = \pm y$.
Substituting this into the Markoff equation gives
\[
4 + 2y^2 = 2y^2
\]
which cannot be. \qed

\emph{Proof of (3)}
The words of length 3 are either $m_2m_3m_2$ or $m_2m_3m_1$, up to permuting the variables $x,y,z$.
Note that $m_2m_3m_2$ is conjugate to $m_3$ since $m_2^{-1}=m_2$, and so it has the same number of fixed points as $m_3$.
For the word $m_2m_3m_1$, all four equations impose nontrivial constraints and we will see that there are no solutions. Composing from left to right, we arrive at
\[
\begin{bmatrix}
(xz-y)(x(xz-y)-z)-x \\
xz- y \\
x(xz-y) - z \\
x^2+y^2+z^2
\end{bmatrix}
=
\begin{bmatrix}
x \\
y \\
z \\
xyz
\end{bmatrix}
\]
The second equation implies $y = xz/2$, and substituting this in the third gives
\[
2z = \frac{x^2z}{2}.
\]
Hence either $z=0$ or $x^2=4$. If $z=0$, then the other equations quickly lead us to the trivial solution $(x,y,z)=(0,0,0)$. Otherwise, we have $x = \pm 2$, and substituting this in the first equation shows that $z^2=4$. Hence any nonzero solutions must be of the form $(x,y,z) = (\pm 2, \pm 2, \pm 2)$. But no such triples solve the Markoff equation $x^2 + y^2 + z^2 = xyz$ since $12 \neq \pm 8$ mod $p$ for any $p \neq 2$. \qed

As an example involving a word of length 4, consider $m_2m_3m_2m_3$. The equation $f=x$ becomes vacuous because the word does not involve $m_1$. The remaining equations $g=y$ and $h=z$ can both be solved by taking $x=0$. Taking $x=0$ in the Markoff equation, we see that every solution of $y^2+z^2=0$ leads to a fixed point. If $p \equiv 1$ mod 4, then $-1$ is a square mod $p$ and any point $(0,y, \sqrt{-1}y)$ is fixed on the Markoff surface. Thus the system (\ref{eqn:fixed}) can have on the order of $p$ solutions even for a word that is not conjugate to any of the Markoff moves.

In all of these examples, there are on the order of $p$ fixed points at the most. Now we present a heuristic suggesting why this trend should continue for longer words, so that the system (\ref{eqn:fixed}) has only $O(p)$ solutions. 
Note first that appying a move such as $h \mapsto fg - h$ at most doubles the overall degree of the polynomials in the sense that, with respect to any of the variables $x$, $y$, or $z$,
\[
\max(\deg(f),\deg(g),\deg(fg-h)) \leq 2\max(\deg(f),\deg(g),\deg(h)).
\]
It could conceivably leave the degree the same if $\deg(h) \geq \deg(f)+\deg(g)$.
In any case, for a word of length $L$, the final $f$, $g$, and $h$ have degree at most $2^L$ in any of the variables $x,y,z$.

Fix $z \in \F_p$. We expect that (\ref{eqn:fixed}) has only $O(1)$ solutions for $x,y$. The Markoff equation amounts to a quadratic in $y$ with at most two solutions per value of $x$:
\[
y = \frac{xz \pm \sqrt{x^2z^2-4(x^2+z^2)}}{2}
\]
We substitute this into any of the other equations in the system (\ref{eqn:fixed}), say $f=x$, to obtain
\begin{equation} \label{eqn:suby}
f\left(x, \frac{xz \pm \sqrt{x^2z^2-4(x^2+z^2)}}{2} ,z\right) = 0
\end{equation}

For a word of length $L$, the polynomial $f(x,y,z)$ has degree at most $2^L$ in any of its variables. Thus, for any fixed $z$, the one-variable polynomial
\[
F_z(x) = (f_{0z}(x) - x)^2 - \big(x^2z^2 - 4(x^2+z^2)\big) f_{1z}(x)
\]
has degree at most $2^{L+1}$. 
Hence (\ref{eqn:suby}) has at most $2^{L+1}$ solutions for $x$, unless $z$ is such that $F_z$ vanishes identically. It is not clear how to rule out this vanishing, and we will take advantage of very special properties of the Markoff surface to do so partially. If there were no such $z$, we could conclude that the number of fixed points is at most $2^{L+1}p$. Instead, our bound will lead to $C^L p$ for some constant $C > 2$, and in particular to $2^{17L+10}p$.
Note also that although the total number of solutions to (\ref{eqn:suby}) might be as high as $2^{L+1}$, perhaps very few of these lie in the ground field $\F_p$.

\section{The Fricke-Klein trace identity and its consequences} \label{sec:fricke}

A helpful interpretation of the Markoff equation, as well as much of the interest in studying it, is provided by Fricke's trace identity:
\[
\tr(A)^2 + \tr(B)^2 + \tr(AB)^2 = \tr(A)\tr(B)\tr(AB) + \tr(ABA^{-1}B^{-1}) + 2
\]
valid for $A,B \in \SL_2(\C)$. As an algebraic identity, this remains valid for $A, B \in \SL_2$ over other rings as well as $\C$. It can be proved using the Cayley-Hamilton theorem and other properties of the trace. See, for instance, Proposition 4.3 in Aigner's book (\cite{A}, p. 65).

If we fix $\tr([A,B]) = -2$ in Fricke's identity, then $(\tr(A),\tr(B),\tr(AB))$ lies on the Markoff surface $M(\C)$.
This can be used to understand the Markoff moves as automorphisms of the free group on two generators.
Write $F_2$ for the free group on two generators $X,Y$ and $\Hom_{\kappa}(F_2, \SL_2(\C))$ for the set of homomorphisms $\theta : F_2 \rightarrow \SL_2(\C)$ such that $\tr(\theta([X,Y])) = \kappa$. 
Since conjugation preserves traces, we have a well-defined map on the quotient by $\SL_2(\C)$-conjugacy:
\[
\Phi: \Hom_{-2}(F_{2},\SL_{2}(\C))/\SL_{2}(\C)\to M(\C)
\]
induced by
\[
\theta \mapsto (\theta(X),\theta(Y),\theta(XY)).
\]
This is a biholomorphism with respect to the natural complex structures on $M(\C)$ and $\Hom_{-2}(F_2, \SL_2(\C))/\SL_2(\C)$. 
The group of outer automorphisms $\Out(F_2)$ acts on $\Hom(F_2,\SL_{2}(\C))$ by composing $\theta: F_2 \rightarrow \SL_2(\C)$ with $\sigma \in \Out(F_2)$. 

If $\tr(\theta[X,Y]) = \kappa$, then we also have
\[
\tr (\theta \circ \sigma [X,Y]) = \tr(\theta([\sigma(X),\sigma(Y)])) = \kappa.
\]
This can be shown using cyclicity of trace together with the fact that $\tr(A) = \tr(A^{-1})$ for $A \in \SL_2(\C)$.
It follows that $\Out(F_2)$ preserves the character variety $\Hom_{-2}(F_2, \SL_2(\C))/\SL_2(\C)$. The map $\Phi$ then induces an action of $\Out(F_2)$ on the Markoff surface $M(\C)$.

Any automorphism preserves the commutator subgroup, and in particular $\Out(F_2)$ acts on the abelianization
\[
F_2^{\text{ab}} = F_2/[F_2,F_2] \cong \Z^2
\]
which is a free abelian group of rank 2. This action induces a map
\[
\Out(F_2) \rightarrow \Aut(\Z^2) = \GL_2(\Z)
\]
and it is a theorem of Nielsen that this map is an isomorphism (see, for instance, Theorem 6.24 in \cite{A} or \cite{N} for the original article).
The kernel is $\{ 1 \}$ because if an automorphism $\sigma$ of $F_2$ acted trivially modulo commutators, it would have to be an inner automorphism (conjugation by a fixed element). If $\overline{\sigma}$ is an automorphism of $F_2/[F_2,F_2]$, then taking $\sigma(w) = \sigma(w[F_2,F_2])$ is well-defined up to inner automorphisms, and so the map $\Out(F_2) \rightarrow \Aut(F_2/[F_2,F_2])$ is onto. Thus we have an isomorphism $\Out(F_2) \cong \GL_2(\Z)$.

The induced action of $\GL_{2}(\Z)$ on $M(\C)$ factors through an action of $\PGL_{2}(\Z)$. This action is by polynomial automorphisms, and these polynomial automorphisms are defined over $\Z$. Therefore they descend to permutations of $M(\F_{p})$.

To determine explicit matrices for the Markoff generators, we argue as follows.
Suppose $X$ is a $2 \times 2$ matrix of determinant 1 (over any ring). By the Cayley-Hamilton theorem, $X$ solves its own characteristic polynomial, so
\[
X^2 - \tr(X)X + 1 = 0.
\]
Multiplying by $YX^{-1}$, we obtain
\[
YX - \tr(X)Y+ YX^{-1} = 0.
\]
Taking the trace of both sides gives
\[
\tr(YX) = \tr(X)\tr(Y) - \tr(YX^{-1}).
\]
This has the same form as a Markoff move on the vector of traces 
\[
(\tr(X),\tr(Y),\tr(YX)),
\]
with the other solution for the third coordinate being $\tr(YX^{-1})$.
To keep the third matrix equal to the product of the first two, we use $\tr(A^{-1}) = \tr(A)$ for $A \in \SL_2$ to rewrite the vector of traces as
\[
(\tr(X),\tr(Y),\tr(YX^{-1}) ) = (\tr(X), \tr(Y^{-1}), \tr(XY^{-1}))
\]
Thus the third move $m_3$ has sent $X$ to $X$ and $Y$ to $Y^{-1}$, which corresponds to the matrix
\[
[m_3] = \begin{bmatrix} 1 & 0 \\ 0 & -1 \end{bmatrix}.
\]
Equally well, since we work in $\PGL_2$, $m_3$ could be represented by $\begin{bmatrix} -1 & 0 \\ 0 & 1 \end{bmatrix}$, which would correspond to writing the trace vector as
\[
(\tr(X),\tr(Y),\tr(YX^{-1}) ) = (\tr(X^{-1}), \tr(Y), \tr(X^{-1} Y) )
\]
by cyclicity of trace.
In the same way, we find that the first move $m_1$ sends $(A,B)$ to $(AB^2,B^{-1})$. The second move sends $A$ to $A^{-1}$ and $B$ to $A^2B$. The third move sends $A$ to $A^{-1}$ and $B$ to $B$. In the abelianization of the free group $\langle A, B \rangle$, these correspond to the matrices
\begin{equation}
[m_1] = \begin{bmatrix}
1 & 0 \\
2 & -1
\end{bmatrix}
[m_2] =
\begin{bmatrix}
-1 & 2 \\
0 & 1
\end{bmatrix},
[m_3] = \begin{bmatrix}
-1 & 0 \\
0 & 1
\end{bmatrix}.
\label{eqn:markoff-matrices}
\end{equation}
In particular, the group generated by these matrices acts on the Markoff surface.
One also has permutations of the three coordinates. For instance, the transposition $\tau_{23}$ acts by
\[
(\tr(A),\tr(B),\tr(AB) ) \mapsto (\tr(A), \tr(AB), \tr(B) ) = ( \tr(A), \tr(A^{-1}B^{-1}), \tr(B^{-1}) )
\]
so that, in matrix form,
\[
[\tau_{23}] = \begin{bmatrix} 1 & -1 \\ 0 & -1 \end{bmatrix}.
\]
Likewise, $[\tau_{13}] = \begin{bmatrix} 1 & 0 \\ 1 & -1 \end{bmatrix}$. As a consistency check, one must have $[m_2] = [\tau_{23} ][m_3][\tau_{23}]$ and $[m_1] = [\tau_{13}][m_3][\tau_{13}]$ up to sign.

\section{The structure of the Markoff group} \label{sec:torsion}

Let $G \leq \PGL_{2}(\Z)$ be the group generated by the involutions
\begin{equation}
m_{1}=\begin{bmatrix}1 & 0\\
2 & -1
\end{bmatrix},\: m_{2}=\begin{bmatrix}-1 & 2\\
0 & 1
\end{bmatrix},\: m_{3}=\begin{bmatrix}-1 & 0\\
0 & 1
\end{bmatrix}.
\end{equation}
These act on $M(\F_{p})$ by polynomial automorphisms known as Markoff moves or Vieta involutions: 
\begin{equation}
m_{1}(x,y,z)=(yz-x,y,z),\:m_{2}(x,y,z)=(x,xz-y,z),\:m_{3}(x,y,z)=(x,y,xy-z).
\end{equation}
As an abstract group, $G\cong\Z/2\Z*\Z/2\Z*\Z/2\Z$ with the $m_{i}$ the generators of the factors in the free product. 
It will follow from this that 
\begin{lemma} \label{lem:The-only-torsion}
The only torsion elements in the Markoff group are the Markoff moves themselves and their conjugates.
\end{lemma}
To show this, we appeal to Kurosh's theorem on the subgroups of free products (see Corollary 4.9.1 in \cite{MKS}). 
\begin{theorem} (Kurosh)
A subgroup $H$ of a free product $A * B * C * \ldots$ must itself be a free product of the form
\[
H = F * \left( *_j g_j H_j g_j^{-1} \right)
\]
where $F$ is free on some subset of $G$ and each $g_j H_j g_j^{-1}$ is conjugate to a subgroup $H_j$ of one of the factors $A, B, C, \ldots$. 
\label{thm:kurosh}
\end{theorem}
In our case, the factors are $A = B = C = \Z/2$ so the only possible subgroups $H_j$ are either trivial or themselves generated by a Markoff move. Suppose $H$ is finite, for instance if $H = \langle w \rangle$ where $w$ is a torsion element. Then the factor $F$ must be trivial, or else $H$ would already be infinite, and likewise there can be only one factor $g_j H_j g_j^{-1}$. Since $H_j$ is either trivial or else equals the whole factor $\langle m \rangle \cong \Z/2$, it follows that $H$ is generated by a single word $gmg^{-1}$ where $m$ is one of the Markoff moves. As claimed, the only torsion elements are conjugates of the Markoff generators.

\section{Bounds for the number of fixed points of words} \label{sec:gl2-fixed}

The goal of this section is to prove the following theorem. 
\begin{theorem} \label{thm:fixed-points}
If $g\in G\leq\PGL_{2}(\Z)$ is not the identity, and is represented by a matrix $\begin{bmatrix}a & b\\
c & d
\end{bmatrix}\in\GL_{2}(\Z)$ 
with $\max(|a|,|b|,|c|,|d|)\leq(p/128)^{1/8}$, then $g$ has at most $1024p\max(|a|,|b|,|c|,|d|)^{8}$ fixed points on $M(\F_{p})$.
\end{theorem}

The exponent 8 on $\max(|a|,|b|,|c|,|d|)$ can likely be replaced by 1. Similarly, the assumption $\max(|a|,|b|,|c|,|d|)\leq(p/128)^{1/8}$ can most likely be loosened. To keep the arguments to their simplest and most readable, and since the bound above is enough for our qualitative result, we chose not to pursue the optimal constants here.

After raising $g$ to a small power, three natural cases arise, and we will give a different bound in each case.

\begin{lemma} \label{lem:power}
For any element $g\in\GL_{2}(\Z)$, there is a power $1\leq K\leq8$ of $g$ such that one of the following holds.
\begin{enumerate}
\item
All the entries of $g^{K}$ have absolute value at least 2.
\item
$g^{K}$ is a torsion element of $\GL_{2}(\Z)$. In this case, $g$ is already torsion.
\item
$g^{K}$ is one of the following types of matrices
\begin{equation}
\pm \begin{bmatrix}1 & n\\
0 & 1
\end{bmatrix},
\pm \begin{bmatrix}1 & 0\\
n & 1
\end{bmatrix},\quad n\in\Z-\{0\}.
\label{eq:parabolic}
\end{equation}
\end{enumerate}
\end{lemma}

\begin{proof}
We first show that one may take $K \leq 4$ in the case $\det(g) = 1$. To avoid considering the case $\det(g)=-1$ separately, we replace $g$ by $g^2$ and double $K$ if necessary. Assuming $\det(g)=1$, the Cayley-Hamilton theorem implies
\[
g^2 - \tr(g)g + I = 0.
\]
If $\tr(g) = 0$, we then have $g^4 = I$ so that $g$ is torsion. If $\tr(g) = \pm 1$, then multiplying by $g$ gives
\[
g^3 = g( \pm g - I ) = \pm (\pm g - I) - g = \mp I
\]
and hence $g^6 = I$. Therefore $g$ is torsion if $| \tr(g)| < 2$. Otherwise, $|\tr(g)| \geq 2$ and we use $ad-bc = 1$ to write
\begin{align*}
g &= \begin{bmatrix} a & b \\ c & d \end{bmatrix} \\
g^2 &= \begin{bmatrix} a(a+d)-1 & b(a+d) \\ c(a+d) & d(a+d)-1 \end{bmatrix} \\
g^4 &= \begin{bmatrix} (a(a+d)-1)^2 + bc(a+d)^2 & b(a+d)( (a+d)^2 - 2) \\ c(a+d)( (a+d)^2 -2) & (d(a+d)-1)^2 + bc(a+d)^2 \end{bmatrix}
\end{align*}
If $bc=0$, then $ad=1$ and $g$ must be of the form (\ref{eq:parabolic}). Otherwise, we have $|b|\geq 1$, $|c| \geq 1$, and $(a+d)^2 \geq 4$. It follows that all entries of $g^4$ are at least 2 in absolute value (moreover, at least 3). The entries of $g^2$ might not be, for instance if $a=0$.
\end{proof}

\subsection{Fixed points of generic elements of $G$}
The `generic' case is when all the entries of $h$ have absolute value $\geq2$. In this case, we use the following bound of Cerbu-Gunther-Magee-Peilen (\cite{CGMP}, Lemma 3.9).
\begin{lemma} \label{lem:cgmp} (Cerbu-Gunther-Magee-Peilen)
If $g=\begin{bmatrix}a & b\\
c & d
\end{bmatrix}\in\GL_{2}(\Z)$ has 
\[
|a|,|b|,|c|,|d| \geq 2,
\]
then $g$ has fewer than $2p\left(|a|+|b|+\big||d|-|c|\big|\right)$ fixed points on $M(\F_{p})$.
\end{lemma}

We refer to \cite{CGMP} for the proof. The assumption that all the entries have absolute value at least 2 makes it possible to implement a rigorous version of the heuristic in Section~\ref{sec:examples}.

\subsection{Fixed points of torsion elements of $G$}
The next lemma bounds the number of fixed points of torsion elements of $G$.
\begin{lemma} \label{lem:torsion-fix}
If $g$ is a non-identity torsion element of $G\leq\PGL_{2}(\Z)$ then $g$ has fewer than $p$ fixed points on $M(\F_{p})$.
\end{lemma}
\begin{proof}
 By Lemma \ref{lem:The-only-torsion}, any non-identity torsion $g$ is conjugate in $G$ to a Markoff move $m_{i}$. Therefore it has the same number of fixed points as $m_{i}$ on $M(\F_{p})$. Since the $m_{i}$ have the same number of fixed points by symmetry, we will count the number of fixed points of $m_{1}$. That is, we count how many points $(x,y,z)\in M(\F_{p})$ satisfy $m_{1}(x,y,z)=(yz-x,y,z)=(x,y,z)$. The last two equations are trivial, and substituting $2x = yz$ into the Markoff surface gives
\[
y^2+z^2 + \frac{y^2z^2}{4} = 2y^2z^2.
\]
For each $y \in \F_p$, this is a quadratic equation for $z$, and hence there are at most $2p$ solutions.
This would be enough for our purposes, but the factor of 2 can be removed by Proposition~\ref{prop:examples}, part (1) (Lemma 2.3, \cite{CGMP}). The actual count is either $p-5$ or $p-3$ and the result follows. 
\end{proof}

\subsection{Fixed points of standard parabolic elements in $G$}
Finally, we estimate the number of fixed points of the standard parabolic elements in the list (\ref{eq:parabolic}).
\begin{proposition} \label{prop:parabolic}
If $g\in\GL_{2}(\Z)$ is one of the matrices $\pm \begin{bmatrix}1 & n\\
0 & 1
\end{bmatrix},\pm \begin{bmatrix}1 & 0\\
n & 1
\end{bmatrix}$ where $0 \neq |n| \leq p$, then $g$ has fewer than $2|n|p$ fixed points on $M(\F_{p})$.
\end{proposition}
\begin{proof}
Consider the ``rotation" $\rot=\begin{bmatrix}1 & 1\\
0 & 1
\end{bmatrix}$, following the notation of Bourgain-Gamburd-Sarnak (\cite{BGS}[2, eq. (3)).
From the matrices in Section~\ref{sec:fricke}, especially,
\[
[m_3] = \begin{bmatrix} 1 & 0 \\ 0 & -1 \end{bmatrix}, [\tau_{23}] = \begin{bmatrix} 1 & -1 \\ 0 & -1 \end{bmatrix}
\]
we see that $\rot = [\tau_{23}] \circ [m_3]$. Again, working in $\PGL_2$, this is to be understood modulo sign.
This element $\rot:(x,y,z)\to(x,xy-z,y)$ thus combines a Markoff move on the third coordinate with a transposition of the second and third coordinates. 
Although $\rot$ itself is not a word in the Markoff moves, we do have
\begin{equation} \label{eqn:rot-square}
\rot^2 = m_2 \circ m_3
\end{equation}
as one sees using $\rot = \tau_{23} \circ m_3$ and $m_2 = \tau_{23} \circ m_3 \circ \tau_{23}$, or observing that both sides send $(x,y,z)$ to $(x, x(xy-z)-y, xy-z)$. 
We will prove Proposition~\ref{prop:parabolic} by carefully examining the orbit structure of $\rot$ on $M(\F_{p})$.
For any fixed $x \in \F_p$, the Markoff equation defines a conic section
\begin{equation}
C(x): \quad y^2 - xyz + z^2 = -x^2.
\end{equation}
Since $\rot$ does not change the first coordinate of $(x,y,z)$, it preserves each of these conics. Its action is given by
\begin{equation}
\rot \begin{bmatrix} y \\ z \end{bmatrix} = \begin{bmatrix} z \\ xz - y \end{bmatrix} = \begin{bmatrix} 0 & 1 \\ -1 & x \end{bmatrix} \cdot \begin{bmatrix} y \\ z \end{bmatrix}
\end{equation}
If $x = \pm 2$, then $C(x)$ degenerates to
\[
C(\pm 2): \quad \left( \frac{y \mp z}{2} \right)^2 = - 1
\]
which is either empty if $p \equiv 3 \bmod 4$ or a pair of lines if $p \equiv 1$. Hence
\[
\# C(\pm 2) = p \left( 1 + \legendre{-1}{p} \right).
\]
For the remaining values of $x$, we have
\[
\# C(x) = p - \legendre{x^2-4}{p}
\]
as we will see by an explicit parametrization. It can also be shown by direct manipulations with the Legendre symbol.

One can think of the conic sections $C(x)$ either as ellipses or hyperbolas modulo $p$ according to whether $x^2-4$ is a square. Following \cite{BGS}, we say $x\in\F_{p}$ is \emph{hyperbolic} if $x^{2}-4$ is a nonzero square in $\F_{p}$. We say $x\in\F_{p}$ is \emph{elliptic} if $x^{2}-4$ is nonzero and not a square. We say $x\in\F_{p}$ is \emph{parabolic} if $x^{2}-4=0$, i.e. $x=\pm2$. 
Note that the parabolic case only arises for $p \equiv 1 \bmod 4$, and that the conic section in such a case is not a parabola but something degenerate.
The behavior of $\rot$ on $C(x)$ was described by Bourgain, Gamburd, and Sarnak in \cite{BGS} using this classification of values of $x$. They state their results for the surface $X^2 + Y^2 + Z^2 = 3XYZ$, although in many of the proofs they use the same normalization $x^2 + y^2 + z^2 = xyz$ as in the present article. The two surfaces are equivalent over $\F_p$ for $p \geq 5$ by a scaling $(X,Y,Z) = (x,y,z)/3$, and we review the corresponding parts of \cite{BGS} for the reader's convenience.

A convenient change of variable toward parametrizing $C(x)$ is
\begin{equation} \label{eqn:xi}
x = \xi + \xi^{-1}
\end{equation}
where $\xi \neq 0$ lies in $\F_p$ if $x^2-4$ is a square, and otherwise in a quadratic extension $\F_{p^2}$.
Let
\begin{equation} \label{eqn:kappa}
\kappa = \kappa(x) = \frac{x^2}{x^2 - 4} = \left( \frac{ \xi + \xi^{-1} }{\xi - \xi^{-1}} \right)^2.
\end{equation}
Then
\[
(x,y,z) = \left(x, t + \frac{\kappa}{t}, t\xi + \frac{\kappa}{t\xi} \right)
\]
solves the Markoff equation for any $t \neq 0$.
Note that multiplying (\ref{eqn:xi}) by $\xi$ gives $\xi^2 - x\xi + 1 = 0$, and this equation simplifies the verification that $(x,t+\kappa t^{-1}, t\xi + \kappa t^{-1}\xi^{-1} )$ solves the Markoff equation.
The action of $\rot$ is to multiply the parameter $t$ by $\xi$. Indeed, from the definition $x = \xi + \xi^{-1}$, we calculate that
\begin{align*}
\rot \left( t + \frac{\kappa}{t}, t\xi + \frac{\kappa}{t\xi} \right) &= \left( t\xi + \frac{\kappa}{t\xi} , xt\xi + \frac{\kappa x}{t\xi} - t - \frac{\kappa}{t} \right) \\
&= \left( t\xi + \frac{\kappa}{t\xi} , t\xi^2+t + \frac{\kappa (1 + \xi^{-2})}{t} - t - \frac{\kappa}{t}  \right) \\
&= \left( t\xi + \frac{\kappa}{t\xi}, t\xi^2 + \frac{\kappa}{t\xi^2} \right)
\end{align*}
that is, $t$ has been multiplied by $\xi$. These considerations can be summarized in the following lemma, due to Bourgain-Gamburd-Sarnak \cite{BGS}.
\begin{lemma} (Bourgain-Gamburd-Sarnak) \label{lem:BGS-lemmas}
\begin{itemize}
\item
(\cite{BGS}, Lemma 3) Let $x$ be parabolic. If $p\equiv3\bmod4$ then $C(x)$ is empty. If $p\equiv1\bmod4$ then $C(x)$ consists of two lines. Letting $i$ be such that $i^{2}\equiv-1\bmod p$, the conic sections are parametrized by
\begin{align*}
C(2) &= (2,t,t \pm 2i) \\
C(-2) &= (-2,t,-t \pm 2i).
\end{align*} 
The action of $\rot$ is given by 
\begin{align*}
\rot(2,t,t\pm2i) &= (2,t\pm2i,t\pm4i), \\
\rot(-2,t,-t\pm2i) &= (-2,-t\pm2i,t\mp4i).
\end{align*}
\item
(\cite{BGS}, Lemma 4) Let $x$ be hyperbolic. Write $x=w+w^{-1}$ with $w\in\F_{p}^{*}$. Let
\[
 \kappa(x)=\frac{x^{2}}{x^{2}-4}.
\]
Then $C(x)$ is parametrized by $\F_{p}^{*}$ via the map 
\[
t\in\F_{p}^{*} \mapsto \left(x,t+\frac{\kappa(x)}{t},tw+\frac{\kappa(x)}{tw}\right).
\]
As a consequence, $|C(x)|=p-1$. After this identification, $\rot$ acts on $C(x)\cong\F_{p}^{*}$ by multiplication by $w$. 
\item
(\cite{BGS}, Lemma 5) Let $x$ be elliptic. Write $x=v+v^{-1}$ where $v\in\F_{p^{2}}-\F_{p}$ and $v^{p+1}=1$. Let $\kappa(x)=\frac{x^{2}}{x^{2}-4}$. Let $E(x)\subset\F_{p^{2}}$ be the set of $t$ such that $t^{p+1}=\kappa(x)$. Then $C(x)$ is parametrized by $E(x)$ via the map
\[
t\in E(x)\mapsto\left(x,t+\frac{\kappa(x)}{t},tv+\frac{\kappa(x)}{tv}\right).
\]
As a consequence, $|C(x)|=p+1$. After this identification, $\rot$ acts on $C(x)\cong E(x)$ by multiplication by $v$. 
\end{itemize}
\end{lemma}
\end{proof}

\emph{Proof of Proposition \ref{prop:parabolic}.}
By multiplying by $-I$, taking inverses, or conjugating by $\begin{bmatrix}0 & 1\\
1 & 0
\end{bmatrix}$, all the matrices of the proposition can be brought into the form $\begin{bmatrix}1 & n\\
0 & 1
\end{bmatrix}$ where $n>0$. 
None of these operations change the number of fixed points of $g$ on $M(\F_{p})$, or the bound for the number of fixed points claimed in the lemma. So it suffices to prove the proposition for $\begin{bmatrix}1 & n\\
0 & 1
\end{bmatrix}$.
This matrix acts on $M(\F_{p})$ by $\rot^{n}$. We split up the fixed points of $\rot^{n}$ depending on whether they belong to $C(x)$ with $x$ parabolic, hyperbolic, or elliptic.
If $x$ is parabolic, Lemma~\ref{lem:BGS-lemmas} implies that for $n<p$, $\rot^{n}$ has no fixed points on $C(x)$.
For fixed hyperbolic $x$, let $x=w+w^{-1}$ as in Lemma~\ref{lem:BGS-lemmas}. Lemma~\ref{lem:BGS-lemmas} implies that $\rot^{n}$ has a fixed point in $C(x)$ if and only if $w^{n}=1$, and this happens if and only if every element of $C(x)$ is fixed by $\rot^{n}$. The number of fixed points of $\rot^{n}$ contained in $C(x)$ with $x$ hyperbolic is therefore bounded by 
\[
\sum_{w\in\F_{p}^{*}:w^{n}=1}|C(w+w^{-1})|=(p-1)|\{w\in\F_{p}^{*}:w^{n}=1\}|\leq(p-1)n.
\]
When $x$ is elliptic, a similar argument using Lemma~\ref{lem:BGS-lemmas} shows that the number of fixed points of $\rot^{n}$ contained in $C(x)$ is bounded by
\[
\sum_{ \substack{v \in \F_{p^{2}}-\F_{p} : \\ v^{p+1}=1,v^{n}=1 } }|C(v+v^{-1})|	=(p+1)|\{v\in\F_{p^{2}}-\F_{p}:v^{p+1}=1,v^{n}=1\}| \leq(p+1)n.
\]
Therefore, when $n<p$, adding our previous bounds together, $\rot^{n}$ has at most $2pn$ fixed points on $M(\F_{p})$. This concludes the proof. \qed

We have used the bound that there are at most $n$ solutions to $w^n = 1$, as a polynomial cannot have more roots than its degree. For many values of $n$, the only solution is $w=1$. Extra solutions arise only if $n$ and $p-1$ have a common factor. This is related to the difficulties encountered in \cite{BGS} when $p^2-1$ has many factors.

\subsection{Proof of Theorem~\ref{thm:fixed-points}.}
Consider any $g \neq 1$ in the Markoff group $G \leq \PGL_2(\Z)$. Let $h = g^K$ where $K \leq 8$ is the power from Lemma~\ref{lem:power}. Any fixed point of $g$ is also a fixed point of its powers, so it suffices to bound the number of fixed points of $h$ on $M(\F_p)$. Note that if $g$ is represented by $\begin{bmatrix} a & b \\ c & d \end{bmatrix}$ in $\GL_2(\Z)$, and $h$ by $\begin{bmatrix} A & B \\ C & D \end{bmatrix}$, then
\[
\max(|A|,|B|,|C|,|D|) \leq 128 \max(|a|,|b|,|c|,|d|)^8.
\]
Indeed, we have
\[
\begin{bmatrix} a & b \\ c & d \end{bmatrix}^2 = \begin{bmatrix} a^2+bc & ab+bd \\ ac+dc & d^2+bc \end{bmatrix}
\]
so that the entries of $g^2$ are at most $2\max(|a|,|b|,|c|,|d|)^2$. 
By repeated squaring, the entries of $g^{2^n}$ are at most $2^{2^n-1} \max(|a|,|b|,|c|,|d|)^{2^n}$.
In particular, for the eighth power we have the bound $128 \max(|a|,|b|,|c|,|d|)^8$, as claimed.
We have assumed that $\max(|a|,|b|,|c|,|d|) \leq (p/128)^{1/8}$ for the express purpose of ensuring that the entries of $h$ are at most $p$. Thus we may apply Proposition~\ref{prop:parabolic}.
This implies that if $h$ is a standard parabolic element, then it has at most 
\[
2p\max(|A|,|B|,|C|,|D|) \leq 256 p \max(|a|,|b|,|c|,|d|)^8
\]
fixed points.
Otherwise, one of the other cases in Lemma~\ref{lem:power} pertains.
In the torsion case, $g$ has at most $p$ fixed points by Lemma~\ref{lem:torsion-fix}. 
This is smaller than the previous bound because
\[
1 \leq \max(|A|,|B|,|C|,|D|)
\]
since the entries of $h$ are integers, not all zero.
In the generic case, Lemma~\ref{lem:cgmp} shows that the number of fixed points of $h$ is at most
\[
2p\big( |A| + |B| + | |C| - |D| | \big) \leq 8p \max(|A|,|B|,|C|,|D|).
\]
Thus in all cases, $h$ has at most
\[
8p \max(|A|,|B|,|C|,|D|) \leq 1024 p \max(|a|,|b|,|c|,|d|)^8
\]
fixed points, and therefore so does $g$.

\section{Proof of the Kesten-McKay Law} \label{sec:wrap-up}

In this section, we prove Theorem~\ref{thm:logp}.
Let $A$ be the adjacency matrix of the Markoff graph and $\lambda_j$ its eigenvalues.
By definition of the empirical measure $\mu_p = \sum \delta_{\lambda_j}$, we have
\[
\int x^L d\mu_p(x) = \sum_j \lambda_j^L = \tr(A^L).
\]
On the other hand, expanding the trace as in Section~\ref{sec:graph} gives
\[
\tr(A^L) =  \sum_{j_1} \cdots \sum_{j_L} a_{j_1,j_2}a_{j_2,j_3} \ldots a_{j_L,j_1}
\]
The product $a_{j_1,j_2}a_{j_2,j_3} \ldots a_{j_L,j_1}$ is 0 unless there is a cycle
\[
j_1 \rightarrow j_2 \rightarrow \ldots \rightarrow j_L \rightarrow j_1
\]
where each arrow represents a Markoff move $m_1, m_2$, or $m_3$. In such a case the product is 1 and the vertex labeled $j_1$ is fixed by some word of length $L$. The trace is obtained by summing over all words
\[
\tr(A^L) = \sum_w \Fix(w) = \sum_{i_1} \cdots \sum_{i_L} \Fix(m_{i_1}\cdots m_{i_L}) 
\]
where $\Fix(w)$ denotes the number of fixed points of $w$ acting on $M(\F_p)$, and the indices $i_1, \ldots, i_L$ take the values 1, 2, 3.
The words that reduce to the identity fix all of $M(\F_p)$ and contribute the main term:
\[
\sum_{\substack{j_1, \ldots j_L \text{s.t.} \\ m_{j_1}\cdots m_{j_L}=1 } } |M(\F_p)| = |M(\F_p)| \int x^L d\rho_3(x) = (p^2 \pm 3p)\int x^L d\rho_3
\]

From the combinatorial interpretation noted in Section~\ref{sec:graph}, the Kesten-McKay moment $\int x^L d\rho_3$ is exactly this count of paths in a tree returning to the starting point.
We will use Theorem~\ref{thm:fixed-points} to show that the remaining words make a negligible contribution, together with the following preparations.
If
\[
g = \begin{bmatrix}
a & b \\
c & d
\end{bmatrix}
\]
is a word of length $L$ in the generators $m_1$, $m_2$, $m_3$, then the entries $a, b, c, d$ are at most exponential in $L$. As an explicit upper bound, we have
\begin{proposition}
The entries of a word of length $L$ in the Markoff moves $m_1,m_2,m_3$ are at most $4^L$ in absolute value.
\label{prop:bound-entries}
\end{proposition}
\begin{proof}
The generators themselves have entries of absolute value at most 2, namely
\[
[m_1] = \begin{bmatrix}
1 & 0 \\
2 & -1
\end{bmatrix},
[m_2] =
\begin{bmatrix}
-1 & 2 \\
0 & 1
\end{bmatrix},
[m_3] = \begin{bmatrix}
-1 & 0 \\
0 & 1
\end{bmatrix}
\]
from equation (\ref{eqn:markoff-matrices}).
This confirms the base case $L=1$ (and would even allow a better exponential rate than $4^L$). For the induction step, consider
\[
\begin{bmatrix} a_{k+1} & b_{k+1} \\ c_{k+1} & d_{k+1} \end{bmatrix}
=
\begin{bmatrix} a & b \\ c & d \end{bmatrix} 
\begin{bmatrix} a_k & b_k\\ c_k & d_k \end{bmatrix} 
=
\begin{bmatrix}
aa_k + bc_k & ab_k+bd_k \\ ca_k + dc_k & cb_k + dd_k
\end{bmatrix}
\]
We have $\max(|a|,|b|,|c|,|d|) \leq 2$ from the base case, $\max(|a_{k}|,|b_{k}|,|c_{k}|,|d_{k}|) \leq 4^k$ from the induction hypothesis, and therefore
\[
\max(|a_{k+1}|,|b_{k+1}|,|c_{k+1}|,|d_{k+1}|) \leq 2 \cdot 4^k + 2 \cdot 4^k = 4^{k+1}.
\]
\end{proof}

\begin{corollary}
There is an absolute exponent $\alpha > 0$ such that if $g \in \GL_2(\Z)$ with $|\tr(g)|>2$ is a word of length $L$ in the matrices representing the Markoff moves $m_1,m_2,m_3$, then $g$ has at most $e^{\alpha L}p$ fixed points.
\end{corollary}
\begin{proof}
By the previous Proposition~\ref{prop:bound-entries}, the entries of $g$ are at most $4^L$ in absolute value. Combining this with Theorem~\ref{thm:fixed-points}, we find that the number of fixed points of $g$ is at most
\[
1024 p (4^L)^8 = p 2^{16L+10}.
\]
Thus we can take $\alpha = 26\log{2} = 18.0218\ldots $ and have the result for all $L \geq 1$.
\end{proof}

Trivially, there are at most $3^L$ words $m_{j_1} \cdots m_{j_L}$ of length $L$ since each index is either 1, 2, or 3. This bound would be enough for our purposes, but a better one is easy enough to come by. The number of words of length $L$ in $\Z/2 * \Z/2 * \Z/2$, or equivalently the number of points at distance $L$ from the root of a 3-regular tree, is $3 \times 2^{L-1}$.
Using the previous corollary over each of these terms leads to
\[
\sum_{i_1} \cdots \sum_{i_L} \Fix(m_{i_1}\cdots m_{i_L}) \leq (3 \times 2^{L-1}) \times (p \times 2^{16L+10})
\]
where the sum is over the remaining words, that is, those that do not reduce to the identity.
Combining this with the main term from the words that do reduce to 1, we have
\[
\tr(A^L) = |M(\F_p)| \int x^L \rho_3(x) dx + O\big( 2^{17L} p \big)
\]
where the implicit constant could be taken as $3 \times 2^9 = 1536$ independent of both $p$ and $L$.
We have $|M(\F_p)| = p^2 \pm 3p$, and normalizing by $p^2$ gives
\[
\int x^L d\mu_p(x) = \int x^L \rho_3(x) dx + O\left( \frac{2^{17L}}{p} \right).
\]
The error term is negligible provided that
\[
L - \frac{\log{p}}{17\log{2}} \rightarrow - \infty.
\]
This allows for $L \sim c\log{p}$ for a sufficiently small $c > 0$, namely
\[
c < \frac{1}{17 \log{2}} = 0.084864\ldots
\]
and one could also take, for instance, $L \sim \frac{1}{17\log{2}} \log{p} - \sqrt{\log{p}}$.

\section{Proof of Corollary~\ref{cor:discrepancy}} \label{sec:discrepancy}

Corollary~\ref{cor:discrepancy} compares the measure of an interval under the empirical distribution of eigenvalues as against the limiting Kesten-McKay law, whereas Theorem~\ref{thm:logp} gives information about moments. 
A natural bridge between these is to approximate the given interval's indicator function by polynomials.
The fact that Theorem~\ref{thm:logp} only allows us to take on the order of $\log{p}$ moments is a handicap compared to, say, having estimates for the Fourier transform $\mu_p(e^{2\pi i \xi x})$. 
There are standard arguments to pass from moments to discrepancy, and in particular Gamburd-Jakobson-Sarnak faced the same problem in a setting very close to ours \cite{GJS}. What we state below as Lemma~\ref{lem:gjs} is a summary of facts given in equation (55) of (\cite{GJS}, Proof of Theorem 1.3). It is based on Selberg polynomials after Selberg (\cite{S}, p.213-219) and Vaaler \cite{V}, as we discuss below. We also recommend Montgomery's treatment (\cite{Mon}, p. 5-15). 

\begin{lemma} \label{lem:gjs}
(Gamburd-Jakobson-Sarnak, after Selberg and Vaaler)
For any interval $I \subseteq [-1,1]$ and $m \in \N$, there exist polynomials $f_{m}^{\pm}$ of degree $m$ such that
\begin{itemize}
\item
$f_{m}^{-} \leq \chi_{I} \leq f_{m}^{+}$ on $[-1,1]$,
\item
There is an absolute constant $B>0$, independent of $I$, such that the coefficients of $f_{m}^{\pm}$ have absolute value $\leq B^{m}$.
\item
\[
\int_{-\frac{1}{\sqrt{2}}}^{\frac{1}{\sqrt{2}}}	(f_{m}^{+}-f_{m}^{-})(y)dy=O\left(\frac{1}{m}\right).
\]
\end{itemize}
\end{lemma}

Approximation by polynomials is closely related to approximation by \emph{band-limited} functions. These are functions $f$ on the real line whose Fourier transform $\widehat{f}$ is supported in an interval $[-L,L]$.
The fact that $f$ does not contain frequencies higher than $L$ is analogous to a polynomial of degree $L$ not involving powers higher than $x^L$. 
Note that, by Poisson summation,
\[
\sum_{n \in \Z} f(x+n) = \sum_{\nu \in \Z} \widehat{f}(\nu) e^{2\pi i \nu x}.
\]
If $\widehat{f}$ is supported in $[-L,L]$, this means that $\sum_n f(x+n)$ is a trigonometric polynomial of degree at most $L$.
Beurling and Selberg studied the problem of majorizing or minorizing a step function by band-limited functions, and their solutions are optimal in the sense of minimizing $L^1$ distance.
Periodizing the Beurling-Selberg functions leads to polynomials on the unit circle, which we identify with the interval $[-3,3]$ in our problem.

Beurling's (unpublished) function $b$ majorizes the signum function $\text{sgn}$, has $\supp \widehat{b} \subseteq [-1,1]$, and minimizes $\int_{\R} b - \text{sgn}$ among all such majorants. It is given explicitly by
\[
b(z) = \left( \frac{\sin{\pi z} }{\pi } \right)^2 \left(\frac{2}{z} + \sum_{n=0}^{\infty} \frac{1}{(z-n)^2} - \sum_{-\infty}^{-1} \frac{1}{(z-m)^2} \right)
\]
and achieves $\int b - \text{sgn} = 1$. 
Note that the factor $\sin(\pi z)^2$ cancels the poles of the other factor, so that $b(z)$ is entire. 
Increasing the bandwidth $L$ leads to a more accurate approximation with $\int b_L - \text{sgn} = 1/L$.
Indeed, let $b_L(x) = b(Lx)$ so that $\widehat{b_L}(\xi) = L^{-1} \widehat{b}(\xi/L)$ is supported in $[-L,L]$.
Since $\text{sgn}(Lx) = \text{sgn}(x)$, we have
\[
\int b_L - \text{sgn} = \frac{1}{L} \int b - \text{sgn} = \frac{1}{L}.
\]

Selberg's majorant must lie above the indicator $\chi_{[\alpha,\beta]}$ rather than the signum function, which can be arranged from Beurling's function using
\[
\chi_{[\alpha,\beta]}(x) = \frac{1}{2} \text{sgn}(x-\alpha) - \frac{1}{2} \text{sgn}(\beta - x).
\]
Thus
\[
S^+(z) = \frac{1}{2} b(z - \alpha) + \frac{1}{2} b(\beta - z)
\]
which obeys $\int S^+ = \beta - \alpha + 1$.
This is optimal when $\beta - \alpha$ is an integer, but not in general.
Selberg also produces a minorant $S^-$, with similar properties except $S^- \leq \chi_{[\alpha, \beta]}$, now an underestimate of $\chi_{[\alpha, \beta]}$.
Increasing the bandwidth to $[-L,L]$ instead of $[-1,1]$ leads to $\int |S_{\pm} - \chi | = 1/L$.

To produce the polynomials we need on $[-3,3]$, we change variables to $\lambda_j = K \cos(2\pi x_j)$. One could take $K=3$ and $0 \leq x_j \leq \pi$, but to avoid problems with periodizing, we take a larger $K$ and $x_j$ in a short interval near $\pi/2$.
Given a subinterval of $[-3,3]$, there is a corresponding interval $[\alpha , \beta]$ of values of $x = \arccos(\lambda/K)$.
Let $S^{\pm} : \R \rightarrow \R$ be the Selberg majorant and minorant for this interval, with Fourier support in $[-1,1]$.
For each $m$, we then let $S_m^{\pm}$ be the scaled majorant or minorant with Fourier support in $[-m,m]$.
Finally, take the trigonometric polynomial
\[
f_m^{\pm}(x) = \sum_{\nu \in \Z} \widehat{S^{\pm}_m}(\nu) e(\nu x) = \sum_{\nu} \widehat{S_m^{\pm}}(\nu) \cos(2\pi \nu x).
\]
Since $\widehat{S_m^{\pm}}(\nu)$ vanishes for $|\nu| > m$, $f_m^{\pm}$ is a polynomial of degree at most $m$ in $\cos(2\pi x)$. Hence it is a polynomial of degree $m$ in $\lambda = K \cos(2 \pi x)$. Explicitly,
\[
f_m^{\pm}(\lambda) = \sum_{\nu} \widehat{S_m^{\pm}}(\nu) T_{\nu}(\lambda/K).
\]
in terms of Chebyshev polynomials $T_{\nu}(\cos{\theta}) = \cos(\nu \theta)$. This can be used to verify that the coefficients are at most exponential in $m$, as claimed in Lemma~\ref{lem:gjs}. For instance, one has the explicit expansion
\[
T_{n}(x) = \frac{\nu}{2} \sum_{k \leq n/2} (-1)^k \frac{(n-k-1)!}{k! (n-2k)!} 2^{n-2k} x^{n-2k}.
\]
The key inequalities that $f_m^+ \geq \chi_I$ and $f_m^{-} \leq \chi_I$ follow from the corresponding properties of $S_m^{\pm}$, together with periodization
\[
f_m^{\pm} = \sum_{n \in \Z} S_m^{\pm}(x+n).
\]
Because we have arranged that the interval of $x$ values has length less than 1, there will only be a single translate $x+n$ in the interval. This is the advantage of choosing a larger $K$ in the change of variables $\lambda = K \cos(2\pi x)$.

\emph{Proof of Corollary \ref{cor:discrepancy}.}
The Markoff eigenvalues lie in $[-3,3]$, and we first rescale so that Lemma~\ref{lem:gjs} applies. Given any subinterval $J$ of $[-3,3]$, let
\[
I = K^{-1} J \subseteq \left[-\frac{1}{\sqrt{2}}, \frac{1}{\sqrt{2}} \right]
\]
where $K = 3\sqrt{2}$.
Let $f_m^{\pm}$ be the polynomials from Lemma~\ref{lem:gjs}, where $m$ will be a small multiple of $\log{p}$, and write
\[
f_m^{\pm}(y) = \sum_{i=0}^m a_{m,i}^{\pm} y^i. 
\]
where $|a_{m,i}^{\pm}| \leq B^m$. We write $\mu_{\infty}$ for the measure with density $\rho_3(x)$ and $\mu_p$ for the eigenvalue counting measure (normalized to have total mass 1).

From Lemma~\ref{lem:gjs}, we have
\begin{equation}
\int f_m^- d\mu_p \leq \mu_p(J) \leq \int f_m^+ d\mu_p.
\label{eqn:above-below}
\end{equation}
By Theorem~\ref{thm:logp},
\[
\int x^i d\mu_p = \int x^i d\mu_{\infty} + O\big( p^{-1} 2^{17 i} \big)
\]
Therefore
\[
\int f_m^{\pm} d\mu_p = \int f_m^{\pm} d\mu_{\infty} + O\left( \sum_{i=0}^m |a_{m,i}^{\pm}| 2^{17i} p^{-1} \right)
\]
Since the coefficients $a_{m,i}^{\pm}$ are at most $B^m$, we have
\[
\int f_m^{\pm} d\mu_p = \int f_m^{\pm} d\mu_{\infty} + O\left((2^{17}B)^m p^{-1} \right)
\]
Therefore we can replace $\mu_p$ by $\mu_{\infty}$ in (\ref{eqn:above-below}):
\[
\int f_m^- d\mu_{\infty} + O\left((2^{17}B)^m p^{-1} \right) \leq \mu_p(J) \leq \int f_m^+ d\mu_{\infty} + O\left((2^{17}B)^m p^{-1} \right).
\]
Using again $f_m^- \leq \chi_J \leq f_m^+$ gives
\[
\int f_m^- d\mu_{\infty} \leq \mu_{\infty}(J) \leq \int f_m^+ d\mu_{\infty}
\]
and since the Kesten-McKay density $\rho_3$ is bounded, Lemma~\ref{lem:gjs} also implies that
\[
\int f_m^+ - f_m^- d\mu_{\infty} \lesssim \frac{1}{m}.
\]
It follows that
\[
|\mu_p(J) - \mu_{\infty}(J) | \lesssim \frac{1}{m} + \frac{(2^{17}B)^m}{p}.
\]
If we choose $m = \lfloor c \log{p} \rfloor$ for a small constant $c > 0$, we obtain
\[
\int f_m^{\pm} d\mu_p = \int f_m^{\pm} d\mu_{\infty} + O\left(p^{-1+c\log(2^{17}B)} \right)
\]
As long as $-1 + c\log(2^{17}B)<0$, this negative power of $p$ can be absorbed in the error $1/m$, which is of order $1/\log{p}$.
Thus
\[
\mu_{p}(J) =\int_{J} \rho_{3}(x) dx+O\left( \frac{1}{\log p} \right)
\]
as required. \qed

\section{Reducing the Markoff group mod $p$} \label{sec:katz?}

Implicit in our claim that a nontrivial word has $O(p)$ fixed points is that it cannot fix all triples $(x,y,z) \in \F_p^3$. In characteristic 0, this follows from the fact that the Markoff group $G = \langle m_1 \rangle * \langle m_2 \rangle * \langle m_3 \rangle$ is a free product. 
We view the Markoff group as polynomial automorphisms of $\Z^3$, and claim reducing a nontrivial element modulo $p$ does not yield the identity polynomial $(x,y,z)$ in characteristic $p$.
Otherwise, it could be written as the identity plus polynomials all of whose coefficients are divisible by $p$.
Such an identity would hold in any field of characteristic $p$, and in particular over $\bar{\F}_p$.
We rule this out as follows.

Let $\overline{G}$ be the subgroup of polynomial transformations of affine space $\A^3(\overline{\F}_p)$ generated by the Markoff moves over an algebraic closure of $\F_p$.
The Markoff moves preserve not only the equation $x^2+y^2+z^2=xyz$, but also any level set $x^2+y^2+z^2=xyz+k$.
In particular, $\overline{G}$ preserves the \emph{Cayley cubic}
\begin{equation} \label{eqn:cayley}
x^2+y^2 +z^2 = xyz + 4.
\end{equation}
The Cayley cubic is special because the action of $\overline{G}$ linearizes after a change of variables. First, let
\begin{align*}
x &= \xi + \xi^{-1} \\
y &= \eta + \eta^{-1}
\end{align*}
where $\xi, \eta \in \overline{\F}_p \setminus \{ 0 \}$.
The solutions of (\ref{eqn:cayley}) for $z$ are then
\[
z_{\pm} = \xi \eta^{\pm 1} + \big( \xi \eta^{\pm 1} \big)^{-1}.
\]
Note that $\xi$ and $\xi^{-1}$ define the same $x$, and similarly the parametrization $\eta \mapsto y$ is two-to-one. 
Suppose that a word $g$ in the Markoff moves acts trivially on the Cayley cubic. We must show $g$ is the identity. We represent $g$ as a matrix $\begin{bmatrix} a & b \\ c & d \end{bmatrix} \in \GL_2(\Z)$ acting on the Cayley cubic by
\[
g(\xi, \eta) = (\xi^a \eta^c, \xi^b \eta^d).
\]
Considering points with $\eta = 1$, the assumption $g(\xi, \eta) = (\xi, \eta)$ implies that $\xi^a = \xi^{\pm 1}$ and $\xi^b = 1$ for all $\xi$ in the algebraic closure $\overline{\F}_p$, not only the finite field $\F_p$. It follows that $a = \pm 1$ and $b = 0$. Likewise, the action on points with $\xi = 1$ implies that $d = \pm 1$ and $c = 0$. Finally, we claim $a=d$, so that $g = \pm I$ is trivial in $\PGL_2(\Z)$. The only other option, up to sign, is to send $(\xi,\eta)$ to $(\xi, \eta^{-1})$. In the original coordinates, this corresponds to $(x,y,z_+) \mapsto (x,y,z_-)$, namely the third Markoff move $m_3$, which certainly does not act trivially. Thus $g$ is the identity. The same argument applies not only to $\overline{\F}_p$ but to any algebraically closed field, and shows that there is no nontrivial relation between the generators $m_1, m_2, m_3$. 
This proves Proposition~\ref{prop:katz}. \qed

We have shown that reduction mod $p$ from $G$ to $\overline{G}$ is injective.  
Equivalently, different words in the Markoff moves define different polynomials, even in characteristic $p$. Note that some of them must coincide as mappings from $\A^3(\F_p)$ to $\A^3(\F_p)$ since there are only finitely many of these, but this will only occur for polynomials of degree larger than $p$.

\section{Conclusion} \label{sec:conc}

We have argued that nontrivial words of length $L$ have at most $p e^{O(L)}$ fixed points, while the identity has $p^2 + O(p)$. 
Thus, for any fixed $L$ or even up to a small multiple of $\log{p}$, the path-count will approximately match what one would get in the process of computing a Kesten-McKay moment. 
The error term $O(p)$ cannot be improved because some words, such as the Markoff moves themselves, do have on the order of $p$ fixed points. There is room for improvement in taking longer words, namely allowing $L$ to be a larger multiple of $\log{p}$.
This would lead to a more refined scale at which the Kesten-McKay holds. 
Beyond the scale $\log{p}$, the Markoff graph no longer resembles a tree in the same statistical sense that we have proved for smaller $L$. To see this, start from the 3-regular tree of integer solutions and reduce mod $p$. There are only $p^2 \pm 3p$ nonzero solutions mod $p$ (and it is not even known whether all of them appear from integer solutions reduced mod $p$). On the other hand, the first $n$ layers in a 3-regular tree comprise
\[
3 \times 2^n - 2
\]
nodes. Once $3 \times 2^n - 2 > p^2+3p$, there must be distinct Markoff triples over $\Z$ that coincide mod $p$. This gives a cycle in $M(\F_p)$ of length at most $2n$ (to the root and back). The same argument produces a loop starting from any solution mod $p$ that lifts to $\Z$, which Bourgain-Gamburd-Sarnak \cite{BGS} prove is the vast majority of them. Thus many cycles of length $4\log_2(p)$ or shorter form as the tree collapses on itself mod $p$.
We would not expect it to be possible to take $L > \frac{4}{\log{2}} \log{p} = (5.77078\ldots ) \log{p}$ and still have agreement with the Kesten-McKay moments. At that scale, if not sooner, loops appear at a positive proportion of the vertices.

The Kesten-McKay law leaves open the question of whether the Markoff graphs are connected for each prime $p$, and the even harder question of whether they form an expander family. The number of connected components of a 3-regular graph is the multiplicity of $\lambda=3$ as an eigenvalue. Corollary~\ref{cor:exceptional-eigs} implies that the number of eigenvalues in an interval $[3-\varepsilon,3]$ is $O(p^2/\log{p})$, which is well short of proving even that the number of components is exactly 1 or even $O(1)$ independent of $p$.
To prove a spectral gap, even if the interval contained a bounded number of eigenvalues, one would need a further argument to rule out some eigenvalues being $3+o(1)$ as $p \rightarrow \infty$. 
The bulk distribution of eigenvalues we have proved here is a coarser property.

\section*{Acknowledgments} 
We thank Seungjae Lee for Figure~\ref{fig:histogram}, the first numerical evidence in favour of the Kesten-McKay law for graphs constructed from the Markoff equation, and many helpful conversations. We thank Peter Sarnak for his encouragement in this project. We thank Nick Katz for raising the question in Section~\ref{sec:katz?}. de Courcy-Ireland's work was supported by the Natural Sciences and Engineering Research Council of Canada through a Postgraduate Scholarships Doctoral grant.

\end{document}